\documentclass[12pt,a4paper]{article}
\usepackage{oldgerm}
\usepackage{amssymb}
\usepackage{amsthm}
\newtheorem{thm}{Theorem}[section]
\newtheorem{prop}[thm]{Proposition}
\newtheorem{lem}[thm]{Lemma}
\newtheorem{defi}[thm]{Definition}

\begin{document}
\title{A rigidity theorem for quaternionic K\"ahler structures}
\author{Kota Hattori}
\date{}
\maketitle
{\begin{center}
{\it Graduate School of Mathematical Sciences, University of Tokyo\\
3-8-1 Komaba, Meguro, Tokyo 153-8914, Japan\\
kthatto@ms.u-tokyo.ac.jp}
\end{center}}
\maketitle
{\abstract We study the moduli space of quaternionic K\"ahler structures on 
a compact manifold of dimension $4n\ge 12$ from a point of view of Riemannian 
geometry, not twistor theory. Then we obtain a rigidity theorem for 
quaternionic K\"ahler structures of nonzero scalar curvature by observing the moduli 
space.}
\section{Introduction}
According to Berger's classification theorem, the holonomy group of a 
simply-connected, non-symmetric, irreducible Riemannian manifold of dimension 
$N$ is isomorphic to one of the following;
\begin{eqnarray}
SO(N),\ U(N/2),\ SU(N/2),\ Sp(N/4),\ Sp(N/4)Sp(1),\ G_2,\ Spin(7).\nonumber
\end{eqnarray}
The Riemannian geometry of special holonomy groups $SU(N/2)$, $Sp(N/4)$, 
$Sp(N/4)Sp(1)$, $G_2$ and $Spin(7)$ are called Calabi-Yau, hyperK\"ahler, 
quaternionic K\"ahler, $G_2$ and $Spin(7)$ structures, respectively. During 
the last score of twentieth century, the deformation theory of these 
structures are studied according to individual way to each structure. For example, 
the deformations of Calabi-Yau and hyperK\"ahler structures were studied by 
using Kodaira-Spencer theory \cite{B}\cite{Ti}\cite{To}. But we cannot apply 
Kodaira-Spencer theory to the other structures since they do not admit complex 
structures. Joyce showed that the moduli spaces of $G_2$ and $Spin(7)$ structures 
are smooth manifolds by studying the closed differential forms which define 
the structures. Then the purpose of this paper is studying the moduli spaces 
of the quaternionic K\"ahler structures.\\
\quad Each quaternionic K\"ahler structure has an Einstein metric. If the 
metric is Ricci-flat, then it reduces to a hyperK\"ahler structure. So, if 
we denote by $\kappa_g$ the scalar curvature of a Riemannian metric $g$, we 
should consider the case of $\kappa_g >0$ or $\kappa_g <0$ for studying 
quaternionic K\"ahler structures.\\
\quad A Riemannian metric $g$ on $4n$-dimensional manifold $M$ is a 
quaternionic K\"ahler metric if the holonomy group of $g$ is isomorphic to a 
subgroup of $Sp(n)Sp(1)$. Then there are rigidity theorems for the quaternionic 
K\"ahler metrics as follows.
\begin{thm}[{\cite{L}}]
Let $M$ be a compact $4n$-manifold, $n\ge 2$, and let $\{g_t\}$ be a family of quaternionic K\"ahler metrics on $M$, of fixed volume, depending smoothly on $t\in\mathbb{R}$. If $g_0$ has positive scalar curvature then there is a family of diffeomorphisms $\{\psi_t\}\subset {\rm Diff}(M)$ depending smoothly on $t\in\mathbb{R}$ such that $\psi_t^*g_t=g_0$.
\end{thm}
Moreover, LeBrun and Salamon \cite{L-S} showed that there are, up to isometries and rescalings, only finitely many compact quaternionic K\"ahler metrics of dimension $4n$ of positive scalar curvature for each $n\ge 2$.
\begin{thm}[{\cite{Hr}}]
Let $M$ be a compact $4n$-manifold and let $\{g_t\}$ be a family of 
quaternionic K\"ahler metrics on $M$, of fixed volume, depending smoothly on 
$t\in\mathbb{R}$. If $g_0$ has negative scalar curvature then there is a family of diffeomorphisms $\{\psi_t\}\subset {\rm Diff}(M)$ depending smoothly on $t\in\mathbb{R}$ such that $\psi_t^*g_t=g_0$.
\end{thm}
The above two theorems are proven by using twistor theory.\\
\quad In this paper, we will prove the rigidity for quaternionic K\"ahler 
structures in the case of $\kappa_g >0$ and $\kappa_g <0$ at the same time 
using Riemannian geometry without using twistor theory.\\
\quad We apply \cite{G} to the description of the moduli spaces of quaternionic 
K\"ahler structures. In \cite{G}, Goto introduced a notion of topological 
calibration which gives a unified framework of the deformation theory of 
Calabi-Yau, hyperK\"ahler, $G_2$ and $Spin(7)$ structures. The moduli space of 
topological calibrations is constructed in Riemannian geometric way.\\
\quad We define the set of quaternionic K\"ahler structures of nonzero 
scalar curvature on $M$ in Section 3 and denote it by 
$\widetilde{\mathcal{M}}_{qK}$ . Since $\widetilde{\mathcal{M}}_{qK}$ is a subset of closed 
$4$-forms on $M$, then $\mathcal{G}:={\rm Diff}_0(M)\times \mathbb{R}_{>0}$ acts on 
$\widetilde{\mathcal{M}}_{qK}$ by the pull-back and scalar multiple. So we 
have a quotient space 
$\mathcal{M}_{qK}:=\widetilde{\mathcal{M}}_{qK}/\mathcal{G}$ and the quotient map $\pi_{qK} 
:\widetilde{\mathcal{M}}_{qK}\to\mathcal{M}_{qK}$. We will show a rigidity theorem for quaternionic K\"ahler 
structures as follows.
\begin{thm}
Let $\{\Phi_t\}_{t\in\mathbb{R}}\in\widetilde{\mathcal{M}}_{qK}$ be a 
continuous family of quaternionic K\"ahler structures on compact 
$4n$-dimensional manifold $M$ for $n\ge 3$. Then we have 
$\pi_{qK}(\Phi_t)=\pi_{qK}(\Phi_0)$ for any $t\in\mathbb{R}$.
\end{thm}
To show Theorem 1.3, we have to evaluate the dimension of the formal 
tangent space of $\mathcal{M}_{qK}$ at $\pi_{qK}(\Phi)$. In Section 2, we 
introduce the deformation complex of quaternionic K\"ahler structures
\begin{eqnarray}
\cdots\stackrel{d}{\longrightarrow}\Gamma(E^k_{\Phi})\stackrel{d}{\longrightarrow}\Gamma(E^{k+1}_{\Phi})\stackrel{d}{\longrightarrow}\cdots.\nonumber
\end{eqnarray}
for each $\Phi\in\widetilde{\mathcal{M}}_{qK}$ along \cite{G}. If we denote the 
$k$-th cohomology group of the above complex by $H^k(\sharp_{\Phi})$, then 
the formal tangent space of $\mathcal{M}_{qK}$ is given by 
$H^1(\sharp_{\Phi})/\mathbb{R}$. To prove the rigidity theorem, we need to show that\\
(I) the deformation complexes are elliptic complexes,\\
(II) $H^1(\sharp_{\Phi})\cong\mathbb{R}$.\\
\quad It is shown that (I) is true in the case of Calabi-Yau, 
hyperK\"ahler, $G_2$ and $Spin(7)$ structures in \cite{G}. But if we try to show (I) in the 
case of quaternionic K\"ahler structures, we have to deal with $4$-forms or 
$5$- forms of $4n$-dimensional vector space, which are so complicated. 
Hence we need more systematic method to study the deformation complex. Then we 
introduce a new method for showing (I) in Section 4.\\
\quad In Section 4, we introduce new complexes called the deformation 
complexes of torsion-free $Sp(n)Sp(1)$-structures, and show they are the 
elliptic complexes. Then we solve (I) by constructing the isomorphisms between the 
deformation complexes of quaternionic K\"ahler structures and new ones. 
Since these discussions can be applied to the other structures, we can regard 
the results in Section 4 as the unified method to study the deformation 
complexes of topological calibrations.\\
\quad We prove (II) in Section 5 by using the Bochner-Weitzenb\"ock formulas and 
vanishing theorems on quaternionic K\"ahler manifolds due to Homma \cite{Ho}, 
Semmelmann and Weingart \cite{S-W}.\\
\quad Each quaternionic K\"ahler structure 
$\Phi\in\widetilde{\mathcal{M}}_{qK}$ induces a Riemannian metric $g_{\Phi}$ on $M$. If $g$ is a 
quaternionic K\"ahler metric of nonzero scalar curvature, then there is a 
quaternionic K\"ahler structure $\Phi\in\widetilde{\mathcal{M}}_{qK}$ such that 
$g_{\Phi}=g$. Then, it is important to study how many quaternionic K\"ahler 
structures which induce a given quaternionic K\"ahler metric $g$. We will 
obtain the following theorem in Section 6.
\begin{thm}
Let $(M,g)$ be a $4n$-dimensional Riemannian manifold for $n\ge 3$ and 
$\widetilde{\mathcal{M}}_{qK}(g):=\{\Phi\in\widetilde{\mathcal{M}}_{qK};\ 
g_{\Phi}=g\}$. If $g$ is a quaternionic K\"ahler metric, then there is a unique element in 
$\widetilde{\mathcal{M}}_{qK}(g)$.
\end{thm}
\section{Geometric structures defined by closed differential forms}
\quad In this section, we introduce Goto's topological calibration theory 
along \cite{G}, then state its relation to torsion-free $G$-structures.\\
\quad Let $g_0$ be the standard inner product on $V=\mathbb{R}^N$. We have 
$GL_N\mathbb{R}$ representation $\rho:GL_N\mathbb{R}\to GL(\Lambda^k)$ by 
putting $\rho(g)\alpha:=(g^{-1})^*\alpha$ for $g\in GL_N\mathbb{R}$ and 
$\alpha\in\Lambda^k:=\Lambda^kV^*$. Fix
\begin{eqnarray}
\Phi^V\in\bigoplus^l_{i=1}\Lambda^{p_i}\nonumber
\end{eqnarray}
such that the isotropy group 
\begin{eqnarray}
G=\{g\in GL_N\mathbb{R};\rho (g)\Phi^V=\Phi^V\}\nonumber
\end{eqnarray}
is a subgroup of the orthogonal group $O(N)$.\\
\quad In this section, we consider a smooth manifold $M$ of dimension $N$. 
We denote by $\pi_{F(M)}:F(M)\to M$ the frame bundle of $M$ whose fibre is 
$GL_N\mathbb{R}$. If we set
\begin{eqnarray}
R_{\Phi^V}(V):=\{\rho (g)\Phi^V\in\bigoplus^l_{i=1}\Lambda^{p_i};\ g\in 
GL_N\mathbb{R}\},\nonumber
\end{eqnarray}
then there is a left action of $GL_N\mathbb{R}$ on $R_{\Phi^V}(V)$, given by $g_1\cdot\rho (g_2)\Phi^V:=\rho (g_1g_2)\Phi^V$ for $g_1,g_2\in GL_N\mathbb{R}$. Then we have an $R_{\Phi^V}(V)$-bundle
\begin{eqnarray}
R_{\Phi^V}(M):=F(M)\times_{GL_N\mathbb{R}}R_{\Phi^V}(V).\nonumber
\end{eqnarray}
Since $R_{\Phi^V}(M)$ is a subbundle of 
$\bigoplus^l_{i=1}\Lambda^{p_i}T^*M$, we can consider the exterior derivative 
$d\Phi\in\bigoplus^l_{i=1}\Omega^{p_i+1}(M)$ for each $\Phi\in\Gamma(R_{\Phi^V}(M))$. Then we put
\begin{eqnarray}
\widetilde{\mathcal{M}}_{\Phi^V}(M)=\{\Phi\in\Gamma (R_{\Phi^V}(M)) ; 
d\Phi=0\}.\nonumber
\end{eqnarray}
By taking proper $N$ and $\Phi^V$, we can construct the set of Calabi-Yau, 
hyperK\"ahler, $G_2$ and $Spin(7)$ structures on $M$ in this manner \cite{G}. We 
will give $\Phi_{qK}\in\Lambda^4(\mathbb{R}^{4n})^*$ which determines the set of quaternionic K\"ahler structures in Section 3.\\
\quad Next we see that there is one-to-one correspondence between torsion-free $G$-structures on $M$ and $\widetilde{\mathcal{M}}_{\Phi^V}(M)$ under a 
certain condition for $\Phi^V$.\\
\quad Since $G$ is a subgroup of $GL_N\mathbb{R}$, we have a quotient space
\begin{eqnarray}
R_G(M):=F(M)/G,\nonumber
\end{eqnarray}
which is a $GL_N\mathbb{R}/G$-bundle over $M$. Then for each section 
$Q\in\Gamma(R_G(M))$, there is a principal $G$-bundle
\begin{eqnarray}
\widetilde{Q}:=\{u\in F(M) ;\ \pi_G(u)=Q_{\pi_{F(M)}(u)}\},\nonumber
\end{eqnarray}
where $\pi_G:F(M)\to F(M)/G$ is the quotient map.\\
\quad By taking a section $Q\in\Gamma(R_G(M))$, we may write 
$TM=\widetilde{Q}\times_G V$ where $V=\mathbb{R}^N$. Then a Riemannian metric $g_{Q}$ on 
$M$ is induced by $g_Q|_p(u\times_Gx,u\times_Gy)=g_0(x,y)$ for $x,y\in V$, $p\in M$ and $u\in\widetilde{Q}_p$. Since $G$ is a subgroup of $O(N)$, this is well-defined.
\begin{defi}
{\rm Let $Q\in \Gamma (R_G(M))$. A covariant derivative $\nabla$ on $TM$ is 
a connection on $Q$ if $\nabla$ is reducible to a connection on a principal 
$G$-bundle $\widetilde{Q}$.}
\end{defi}
\begin{defi}
{\rm We call $Q\in \Gamma (R_G(M))$ a torsion-free $G$-structure if the 
Levi-Civita connection of $g_Q$ is a connection on $Q$.}
\end{defi}
The natural diffeomorphism $R_{\Phi^V}(V)\cong GL_N\mathbb{R}/G$ induces a 
bijective bundle map
\begin{eqnarray}
\sigma_{\Phi^V}:R_{\Phi^V}(M)\longrightarrow R_G(M).\nonumber
\end{eqnarray}
We put $Q_{\Phi}:=\sigma_{\Phi^V}(\Phi)\in\Gamma (R_G(M))$ for each $\Phi\in\Gamma (R_{\Phi^V}(M))$.\\
\quad We set a $G$-equivariant map $A^k_{\Phi_V}:\Lambda^k\otimes V\to\bigoplus^l_{i=1}\Lambda^{p_i+k-1}$ as $A^k_{\Phi_V}(\omega\otimes v):=\omega\wedge \iota_v \Phi^V$ for $\omega\otimes v\in \Lambda^k\otimes V$, where $\iota$ is the interior product, and put $E^k_{\Phi^V}:=Im(A^k_{\Phi_V})$. Then a bundle map $A^k_{\Phi}:\Lambda^kT^*M\otimes TM\to E^k_{\Phi}$ is induced by $A^k_{\Phi_V}$ for each $\Phi\in\Gamma (R_{\Phi^V}(M))$, where we put $E^k_{\Phi}:=\widetilde{Q_{\Phi}}\times_G E^k_{\Phi^V}$.
\begin{prop}
Let $\nabla$ be a connection on $Q_{\Phi}$ for $\Phi=(\Phi_1, \cdots , 
\Phi_l)\in\Gamma (R_{\Phi^V}(M))$. Then we have
\begin{eqnarray}
d\Phi\ =\ A^2_{\Phi}(T^{\nabla}),\nonumber
\end{eqnarray}
where $T^{\nabla}$ is the torsion tensor of $\nabla$.
\end{prop}
\begin{proof}
We calculate $(d\Phi)_p$ for a fixed point $p\in M$. Let $\nabla$ be any 
connection on $Q_{\Phi}$, $v_1,v_2,\cdots,v_N\in V$ be an orthonormal basis 
and $v^1,v^2,\cdots,v^N\in V^*$ be its dual basis.\\
\quad We can take a neighborhood $U$ of $p$ and local section $\tau\in\Gamma(U,\widetilde{Q_{\Phi}})$ which satisfy $(\nabla\xi_i)_p=0$, where $\xi_i|_x=\tau(x)\times_Gv_i$ for $x\in U$.\\
\quad Let $\Phi^V_l = \sum_{i_1,\cdots,i_{p_l}}\Phi^{(l)}_{i_1,\cdots,i_{p_l}} 
v^{i_1}\wedge\cdots\wedge v^{i_{p_l}}\quad (l=1,\cdots ,N)$ and 
$\xi^i|_x=\tau(x)\times_Gv^i$. Then for any $x\in U$, we have
\begin{eqnarray}
(\Phi_l)_x&=&\sigma(x)\times_G \Phi^V_l\nonumber\\
&=&\ \sum_{i_1,\cdots,i_{p_l}}\Phi^{(l)}_{i_1,\cdots,i_{p_l}} 
(\xi^{i_1})_x\wedge\cdots\wedge (\xi^{i_{p_l}})_x.\nonumber
\end{eqnarray}
\quad If we put $d\xi^{\alpha}=c^{\alpha}_{\beta\gamma}\xi^{\beta}\wedge 
\xi^{\gamma}$, where $c^{\alpha}_{\beta\gamma}$ are smooth function on $U$, 
then $[\xi_{\beta},\xi_{\gamma}]=-c^{\alpha}_{\beta\gamma}\xi_{\alpha}$. So 
we have
\begin{eqnarray}
(d\Phi_l)_p
&=&\sum_{i_1,\cdots,i_{p_l}}\sum_{s =1}^{p_l}(-1)^{s 
-1}\Phi^{(l)}_{i_1,\cdots,i_{p_l}} (\xi^{i_1})_p\wedge\cdots\wedge(d\xi^{i_{s}})_p\wedge\cdots\wedge 
(\xi^{i_{p_l}})_p\nonumber\\
&=&\sum_{i_1,\cdots,i_{p_l}}\sum_{s 
=1}^{p_l}\sum_{\beta,\gamma}\Phi^{(l)}_{i_1,\cdots,i_{p_l}} \{c^{i_s}_{\beta\gamma}\xi^{\beta}\wedge 
\xi^{\gamma}\wedge\iota_{\xi_{i_s}}(\xi^{i_1}\wedge\cdots\wedge 
\xi^{i_{p_l}})\}_p\nonumber\\
&=&\sum_{i_1,\cdots,i_{p_l}}\sum_{\alpha,\beta,\gamma}\Phi^{(l)}_{i_1,\cdots,i_{p_l}} 
\{c^{\alpha}_{\beta\gamma}\xi^{\beta}\wedge 
\xi^{\gamma}\wedge\iota_{\xi_{\alpha}}(\xi^{i_1}\wedge\cdots\wedge \xi^{i_{p_l}})\}_p\nonumber\\
&=&A^2_{\Phi}(\sum_{\alpha,\beta,\gamma}c^{\alpha}_{\beta\gamma}\xi^{\beta}\wedge 
\xi^{\gamma}\otimes{\xi_{\alpha}})_p\nonumber
\end{eqnarray}
Since we may write
\begin{eqnarray}
T^{\nabla}(\xi_{\beta},\xi_{\gamma})_p &=& 
-[\xi_{\beta},\xi_{\gamma}]_p\nonumber\\
&=& -c^{\alpha}_{\beta\gamma}(\xi_{\alpha})_p,\nonumber
\end{eqnarray}
then we have $d\Phi_p=A^2_{\Phi}(T^{\nabla})_p$ for any $p\in M$
\end{proof}
Note that Lie group $G$ acts on $\textswab{g}:=Lie(G)\subset 
End(V)=V^*\otimes V$ by the adjoint action. Let $v_1,v_2,\cdots ,v_N\in V$ be a basis and 
$v^1,v^2,\cdots ,v^N\in V$ be its dual basis. For each $Q\in \Gamma 
(R_G(M))$, there is a sub vectorbundle 
$\hat{\textswab{g}}^k_Q:=\widetilde{Q}\times_G\textswab{g}^k$ of $\Lambda^kT^*M\otimes TM$, where
\begin{eqnarray}
\textswab{g}^k:=span\{\sum_{i,j}(\alpha\wedge a_i^jv^i)\otimes 
v_j;\alpha\in\Lambda^{k-1} , \sum_{i,j}a_i^jv^i\otimes 
v_j\in\textswab{g}\}\subset\Lambda^k\otimes V\nonumber
\end{eqnarray}
for $k\ge 2$, and
\begin{eqnarray}
\textswab{g}^1:=\textswab{g},\quad\textswab{g}^0:=\{0\}.\nonumber
\end{eqnarray}
Then we have an orthogonal decomposition $\Lambda^k\otimes 
V=\textswab{g}^k\oplus P^k_\textswab{g}$ where $P^k_\textswab{g}$ is the orthogonal 
complement. If we put $\hat{P}^k_Q:=\widetilde{Q}\times_G P^k_\textswab{g}$, we 
have an orthogonal decomposition
\begin{eqnarray}
\Lambda^kT^*M\otimes TM=\hat{\textswab{g}}^k_Q\oplus \hat{P}^k_Q\nonumber
\end{eqnarray}
with respect to $g_Q$.\\
\quad Let $\bar{A}^k_{\Phi_V}:=A^k_{\Phi_V}|_{P^k_{\textswab{g}}}$. It is 
clear that $\textswab{g}^k$ is a subspace of $Ker(A^k_{\Phi^V})$ from the 
definitions of $\textswab{g}^k$ and $A^k_{\Phi^V}$. If we assume that $dim 
E^k_{\Phi^V}=dim P^k_{\textswab{g}}$, then the induced bundle map 
$\bar{A}^k_{\Phi}:\hat{P}^k_{Q_{\Phi}}\to E^k_{\Phi}$ is an isomorphism for each $\Phi\in\Gamma (R_{\Phi^V}(M))$.
\begin{prop}[{\cite{S1}}]
We define a linear map $\mathbf{a}:V^*\otimes End(V)\ \to \ 
\Lambda^2\otimes V$ by
\begin{eqnarray}
\mathbf{a}(u_1\otimes u_2\otimes v):=u_1\wedge u_2\otimes v=(u_1\otimes 
u_2-u_2\otimes u_1)\otimes v\nonumber
\end{eqnarray}
for $u_1,\in V^*$, $u_2\otimes v\in V^*\otimes V=End(V)$.\\
Then $\mathbf{a}|_{V^*\otimes \mathbf{so}(N)}:V^*\otimes \mathbf{so}(N)\to 
\Lambda^2\otimes V$ is an isomorphism, where $\mathbf{so}(N)$ is the Lie 
algebra of $O(N)$.
\end{prop}
See Proposition 2.1 of \cite{S1} as to the proof.
\begin{prop}
Let $\Phi\in\Gamma (R_{\Phi^V}(M))$ and suppose 
$dimE^2_{\Phi^V}=dimP^2_{\textswab{g}}$. Then $Q_{\Phi}$ is a torsion-free $G$-structure if $d\Phi = 0$.
\end{prop}
\begin{proof}
\quad Let $\nabla$ be a connection on $Q_{\Phi}$ and assume that $d\Phi = 
0$. We may write $\nabla=\nabla^{\Phi}+\gamma$, where $\gamma$ is a section 
of $\widetilde{Q_{\Phi}}\times_G(V^*\otimes \mathbf{so}(N))$ and 
$\nabla^{\Phi}$ is the Levi-Civita connection of $g_{\Phi}:=g_{Q_{\Phi}}$. Then
\begin{eqnarray}
T^{\nabla}=T^{\nabla^{\Phi}}+\mathbf{a}(\gamma)=\mathbf{a}(\gamma).\nonumber
\end{eqnarray}
Since we have $A^2_{\Phi}(T^{\nabla})=d\Phi=0$ from the assumption and 
Proposition 2.3, then
\begin{eqnarray}
\mathbf{a}(\gamma)=T^{\nabla}\in 
Ker(A^2_{\Phi})=\Gamma(\hat{\textswab{g}}^2_{Q_{\Phi}}).\nonumber
\end{eqnarray}
Then $\gamma$ is a section of $\widetilde{Q_{\Phi}}\times_G(V^*\otimes \textswab{g})$ 
from Proposition 2.4, which means 
$\gamma\in\Omega^1(\hat{\textswab{g}}^1_{Q_{\Phi}})$. Hence we have shown that the Levi-Civita connection 
$\nabla^{\Phi}=\nabla-\gamma$ is a connection on $Q_{\Phi}$.
\end{proof}
\begin{thm}
Let $\nabla$ be the Levi-Civita connection of $g_{\Phi}$ for a section 
$\Phi\in\Gamma (R_{\Phi^V}(M))$. We suppose 
$dimE^2_{\Phi^V}=dimP^2_{\textswab{g}}$. Then the following conditions are equivalent.\\
(i)$d\Phi=0$. \ (ii)$Q_{\Phi}$ is a torsion-free $G$-structure. \ 
(iii)$\nabla\Phi=0$.
\end{thm}
\begin{proof}
Proposition 2.5 gives (i)$\Longrightarrow$(ii). \\
\quad Assume that $Q_{\Phi}$ is a torsion-free $G$-structure. Then the 
Levi-Civita connection $\nabla$ is a connection on $Q_{\Phi}$. If we take $p\in M$, 
$U$ and $\xi^1,\cdots,\xi^n$ as in Proposition 2.3, then 
$(\nabla\xi^i)_p=0$. So 
\begin{eqnarray}
(\nabla\Phi_i)_p &=& \sum_{j_1,\cdots,j_{p_i}}\sum_{s 
=1}^{p_i}\Phi_{j_1,\cdots,j_{p_i}}^i (\xi^{j_1})_p\wedge\cdots\wedge(\nabla 
\xi^{j_{s}})_p\wedge\cdots\wedge (\xi^{j_{p_i}})_p\nonumber\\
&=& 0.\nonumber
\end{eqnarray}
Thus we have shown $\nabla\Phi=0$ if $Q_{\Phi}$ is a torsion-free 
$G$-structure.\\
\quad If we assume $\nabla\Phi=0$, then 
$d\Phi=\sum_j\xi^j\wedge\nabla_{\xi_j}\Phi=0$.
\end{proof}
Next we consider the deformation complex of 
$\Phi_0\in\widetilde{\mathcal{M}}_{\Phi^V}(M)$.
\begin{prop}[{\cite{G}}]
Let $\Phi_0\in\widetilde{\mathcal{M}}_{\Phi^V}(M)$. Then 
$d\Gamma(E^k_{\Phi_0})$ is a subspace of $\Gamma(E^{k+1}_{\Phi_0})$.
\end{prop}
From Proposition 2.7, we obtain Goto's complex
\begin{eqnarray}
\cdots\stackrel{d}{\longrightarrow}\Gamma(E^k_{\Phi_0})\stackrel{d}{\longrightarrow}\Gamma(E^{k+1}_{\Phi_0})\stackrel{d}{\longrightarrow}\cdots.
\end{eqnarray}
\section{Moduli spaces of the quaternionic K\"ahler structures}
\quad In this section, we state the main result in this paper. First, we 
define quaternionic K\"ahler structures and their moduli space.\\
\quad From now on we consider the case of $N=4n$. We put $I,J,K\in 
M_{4n}\mathbb{R}$ be almost complex structures on $V$ defined by
\[ I=
\left(
\begin{array}{cccc}
O & -E_n & O & O \\
E_n & O & O & O \\
O & O & O & -E_n \\
O & O & E_n & O 
\end{array}
\right)
\quad 
J=
\left(
\begin{array}{cccc}
O & O & -E_n & O \\
O & O & O & E_n \\
E_n & O & O & O \\
O & -E_n & O & O 
\end{array}
\right)
\]
\[ K=
\left(
\begin{array}{cccc}
O & O & O & -E_n \\
O & O & -E_n & O \\
O & E_n & O & O \\
E_n & O & O & O 
\end{array}
\right)
\]
where $E_n$ is the unit matrix of $GL_n\mathbb{R}$, and set
\begin{eqnarray}
\omega_I:=g_0(I\cdot,\cdot),\quad \omega_J:=g_0(J\cdot,\cdot),\quad 
\omega_K:=g_0(K\cdot,\cdot).\nonumber
\end{eqnarray}
Then we have an 4-form
\begin{eqnarray}
\Phi_{qK}:=\omega_I\wedge\omega_I + \omega_J\wedge\omega_J + 
\omega_K\wedge\omega_K\in\Lambda^4,\nonumber
\end{eqnarray}
whose isotropy group
\begin{eqnarray}
G_{\Phi_{qK}}:=\{g\in GL_N\mathbb{R};\ \rho 
(g)\Phi_{qK}=\Phi_{qK}\}\nonumber
\end{eqnarray}
is equal to $Sp(n)Sp(1):=Sp(n)\times_{\{\pm 1\}}Sp(1)$.
\begin{defi}
{\rm Let $M$ be a smooth manifold of dimension $4n$. Then we call 
$\Phi_0\in\Gamma(R_{\Phi_{qK}}(M))$ is a quaternionic K\"ahler structure on $M$ if 
and only if $d\Phi_0=0$.}
\end{defi}
There is the irreducible decompositions of $Sp(n)Sp(1)$-representation 
according to \cite{F}\cite{S1},
\begin{eqnarray}
\Lambda^3\otimes\mathbb{C} &=& 
\lambda^3_0\sigma^3\oplus\lambda^1_0\sigma^3\oplus\lambda^3_1\sigma^1\oplus\lambda^1_0\sigma^1\nonumber,\\
\Lambda^4\otimes\mathbb{C} &=& 
\lambda^4_0\sigma^4\oplus\lambda^2_0\sigma^4\oplus\sigma^4\oplus\lambda^4_1\sigma^2\oplus\lambda^2_1\sigma^2\oplus\lambda^2_0\sigma^2\oplus\lambda^4_2\oplus\lambda^2_0\oplus\sigma^0,\nonumber\\
\Lambda^5\otimes\mathbb{C}&=&\lambda^5_0\sigma^5\oplus 
\lambda^3_0\sigma^5\oplus \lambda^1_0\sigma^5\oplus \lambda^5_1\sigma^3\oplus 
\lambda^3_1\sigma^3\oplus \lambda^5_2\sigma^1\oplus \lambda^3_0\sigma^1\oplus 
\Lambda^3\otimes\mathbb{C}.\nonumber
\end{eqnarray}
Here we write $\lambda^p_q\sigma^r=\lambda^p_q\otimes\sigma^r$ where 
$\lambda^p_q$ is an irreducible $Sp(n)$-module and $\sigma^r=S^r(\mathbb{C}^2)$ 
is an irreducible $Sp(1)$-module. The representation $\lambda^p_q$ has the 
highest weight $(\mu_1,\cdots,\mu_n)$ such that
\[ \mu_l=
\left\{
\begin{array}{ccc}
2 & 1\le l\le q, \\
1 & q+1\le l\le p-q, \\
0 & p-q+1\le l\le n.
\end{array}
\right.
\]
Then we have
\begin{eqnarray}
E^0_{\Phi_{qK}}\otimes\mathbb{C} &=& \lambda^1_0\sigma^1,\nonumber\\
E^1_{\Phi_{qK}}\otimes\mathbb{C} &=& 
\lambda^2_0\sigma^2\oplus\lambda^2_1\sigma^2\oplus\lambda^2_0\oplus\sigma^0,\nonumber
\end{eqnarray}
by the definition of $\Phi_{qK}$ and direct calculation. As to 
$E^2_{\Phi_{qK}}$, there is the irreducible decomposition for $n\ge 3$
\begin{eqnarray}
E^2_{\Phi_{qK}}\otimes\mathbb{C} = \lambda^3_1\sigma^3\oplus 
\lambda^3_0\sigma^1\oplus \Lambda^3\otimes\mathbb{C}.\nonumber
\end{eqnarray}
by \cite{Sw}.\\
\quad Weyl dimension formula \cite{F-H} of $Sp(n)$ reperesentation gives
\begin{eqnarray}
dim_{\mathbb{C}}\ \lambda^p_q = \frac{2^n n! \prod_{1\le i,j\le 
n}(\tilde{\mu}_i - \tilde{\mu}_j)(\tilde{\mu}_i + \tilde{\mu}_j + 2n + 
2)\prod_{k=1}^n(\tilde{\mu}_k + n + 1)}{\prod_{k=1}^n (2k)!}.\nonumber
\end{eqnarray}
Then we can calculate the dimension of $E^2_{\Phi_{qK}}$.
\begin{thm}[{\cite{Sw}}]
Let $M$ be a $4n$-dimensional manifold for $n\ge 3$ and 
$\Phi_0\in\Gamma(R_{\Phi_{qK}}(M))$. Then the Levi-Civita connection $\nabla^{\Phi_0}$ of 
$g_{\Phi_0}$ reduces to the connection of the principal $Sp(n)Sp(1)$-bundle 
$Q_{\Phi_0}$ if and only if $d\Phi_0=0$.
\end{thm}
\begin{proof}
It suffices to show that $dim E^2_{\Phi_{qK}} = dim 
P^2_{\mathbf{sp}(n)\oplus\mathbf{sp}(1)}$ from Theorem 2.6. By Weyl dimension formula, we have
\begin{eqnarray}
dim E^2_{\Phi_{qK}} = 24n^3 - 12n^2 - 12n,\nonumber
\end{eqnarray}
for $n\ge 3$. From $dim (\mathbf{sp}(n)\oplus\mathbf{sp}(1)) = 2n^2+n+3$ 
and the injectivity of $\mathbf{a}|_{V^*\otimes \mathbf{so}(4n)}$, we have
\begin{eqnarray}
dim P^2_{\mathbf{sp}(n)\oplus\mathbf{sp}(1)} &=& dim (\Lambda^2\otimes V) - 
dim \{V^*\otimes (\mathbf{sp}(n)\oplus\mathbf{sp}(1))\}\nonumber\\
 &=& 24n^3 - 12n^2 - 12n.\nonumber
\end{eqnarray}
\end{proof}
Now we denote the scalar curvature of a Riemannian metric $g$ by 
$\kappa_g$. If the holonomy group $Hol(g)$ is isomorphic to the subgroup of 
$Sp(n)Sp(1)$, then $g$ is Einstein, so $\kappa_g$ is constant \cite{S1}. Moreover 
$Hol(g)$ is isomorphic to $Sp(n)Sp(1)$ if and only if $\kappa_{g_{\Phi_0}}\neq 0$. 
So we put
\begin{eqnarray}
\widetilde{\mathcal{M}}_{qK}:=\{\Phi_0\in\widetilde{\mathcal{M}}_{\Phi_{qK}}(M);\ 
\kappa_{g_{\Phi_0}}\neq 0\}.\nonumber
\end{eqnarray}
Let $\mathcal{G}:={\rm Diff}_0(M)\times\mathbb{R}_{>0}$ where ${\rm 
Diff}_0(M)$ is the identity component of ${\rm Diff}(M)$. Then $\mathcal{G}$ acts 
on $\widetilde{\mathcal{M}}_{qK}$ by putting $(f,c)\cdot\Phi:=cf^*\Phi$ for 
$(f,c)\in\mathcal{G}$. Thus we obtain the moduli spaces of quaternionic 
K\"ahler structures of nonzero scalar curvature
\begin{eqnarray}
\mathcal{M}_{qK}:=\widetilde{\mathcal{M}}_{qK}/\mathcal{G}.\nonumber
\end{eqnarray}
\quad Next we show the rigidity theorem for quaternionic K\"ahler 
structures of nonzero scalar curvature. From now on, we suppose that $M$ is compact. 
We will use following lemmas.
\begin{lem}
Let $M$ be a compact manifold of dimension $4n\ge 12$. Then Goto's complex 
of quaternionic K\"ahler structures
\begin{eqnarray}
\cdots\stackrel{d}{\longrightarrow}\Gamma(E^k_{\Phi_0})\stackrel{d}{\longrightarrow}\Gamma(E^{k+1}_{\Phi_0})\stackrel{d}{\longrightarrow}\cdots\nonumber
\end{eqnarray}
is elliptic complex at $k=1$ for each 
$\Phi_0\in\widetilde{\mathcal{M}}_{\Phi_{qK}}(M)$. In particular, there is the Hodge decomposition
\begin{eqnarray}
\Gamma(E^1_{\Phi_0})=\mathbb{H}^1_{\Phi_0}\oplus 
d\Gamma(E^0_{\Phi_0})\oplus d^*_1\Gamma(E^2_{\Phi_0}),\nonumber
\end{eqnarray}
where $d^*_k$ is a formal adjoint operator of 
$d:\Gamma(E^k_{\Phi_0})\to\Gamma(E^{k+1}_{\Phi_0})$ and $\mathbb{H}^1_{\Phi_0}$ is given by
\begin{eqnarray}
\mathbb{H}^1_{\Phi_0}:=Ker(\triangle_{\sharp}:=dd^*_0+d^*_1d:\Gamma(E^1_{\Phi_0})\to\Gamma(E^1_{\Phi_0})).\nonumber
\end{eqnarray}
\end{lem}
\begin{lem}
Let $M$ be a compact manifold of dimension $4n\ge 12$. Then we have
\begin{eqnarray}
\mathbb{H}^1_{\Phi_0}=\mathbb{R}\Phi_0\nonumber
\end{eqnarray}
for each $\Phi_0\in\widetilde{\mathcal{M}}_{qK}$.
\end{lem}
We will prove Lemma 3.3 and 3.4 in Section 4 and 5, respectively.\\
\quad Let $F$ be a fibre bundle over $M$ and $k\ge 2n+1$. Then an 
$L^2_k$-section of $F$ is a $C^0$-section by Sobolev embedding theorem. By putting 
$(f,c)\cdot\Phi:=cf^*\Phi\in L^2_k(R_{\Phi_{qK}}(M))$ for $(f,c)\in 
L^2_{k+1}({\rm Diff}_0(M))\times\mathbb{R}_{>0}$ and $\Phi\in 
L^2_k(R_{\Phi_{qK}}(M))$, an infinite dimensional Lie group $\mathcal{G}_{k+1}:=L^2_{k+1}({\rm 
Diff}_0(M))\times\mathbb{R}_{>0}$ acts on a Hilbert manifold 
$L^2_k(R_{\Phi_{qK}}(M))$. Thus we have a quotient topological space
\begin{eqnarray}
\mathcal{A}_k:=L^2_k(R_{\Phi_{qK}}(M))/\mathcal{G}_{k+1}\nonumber
\end{eqnarray}
and the quotient map $\pi_k:L^2_k(R_{\Phi_{qK}}(M))\to\mathcal{A}_k$. Then 
we are going to show that $\pi_k(\widetilde{\mathcal{M}}_{qK})$ is a 
discrete subset of $\mathcal{A}_k$ for $k\ge 2n+1$ and $n\ge 3$. This is proven 
directly from Proposition 3.7.\\
\quad First we consider the neighborhood of $\pi_k(\Phi_0)\in\mathcal{A}_k$ for $\Phi_0\in\widetilde{\mathcal{M}}_{\Phi_{qK}}(M)$. Since $\bar{A}^1_{\Phi_0}:\hat{P}^1_{Q_{\Phi_0}}\to E^1_{\Phi_0}$ is an 
isomorphism, there is the inverse map $(\bar{A}^1_{\Phi_0})^{-1}$. Then we 
define a map $\varphi_{k,\Phi_0}:L^2_k(E^1_{\Phi_0})\to 
L^2_k(R_{\Phi_{qK}}(M))$ by 
$\varphi_{k,\Phi_0}(\alpha):=\rho(e^{(\bar{A}^1_{\Phi_0})^{-1}\alpha})\Phi_0$ for $\Phi_0\in\widetilde{\mathcal{M}}_{\Phi_{qK}}(M)$ and 
$\alpha\in L^2_k(E^1_{\Phi_0})$, where $\rho(g)\Phi_0=(g^{-1})^*\Phi_0$ for $g\in 
L^2_k(GL(TM))$ and $e^X:=\sum^{\infty}_{j=0}X^j/j!$ for $X\in L^2_k(End\ 
TM)$. The differential of the map $\varphi_{k,\Phi_0}$ at the origin is given 
by
\begin{eqnarray}
(\varphi_{k,\Phi_0*})_0(\beta)=-\bar{A}^1_{\Phi_0}(\bar{A}^1_{\Phi_0})^{-1}(\beta)=-\beta\nonumber
\end{eqnarray}
for $\beta\in L^2_k(E^1_{\Phi_0})$. If we put
\begin{eqnarray}
U_{k,\Phi_0}(\varepsilon):=\{\alpha\in L^2_k(E^1_{\Phi_0});\ 
\|\alpha\|_{L^2_k}<\varepsilon\}\nonumber
\end{eqnarray}
for $\varepsilon >0$, then there is $\varepsilon >0$ such that
\begin{eqnarray}
\varphi_{k,\Phi_0}\big|_{U_{k,\Phi_0}(\varepsilon)}:U_{k,\Phi_0}(\varepsilon)\longrightarrow 
\varphi_{k,\Phi_0}(U_{k,\Phi_0}(\varepsilon))\nonumber
\end{eqnarray}
is a diffeomorphism from inverse function theorem.\\
\quad Set $V_{k,\Phi_0}(\varepsilon):=\{\alpha\in 
U_{k,\Phi_0}(\varepsilon);\ d^*_0\alpha=0,\ <\alpha,\Phi_0>_{L^2(\Phi_0)}=0\}$ where 
$<\alpha,\Phi_0>_{L^2(\Phi_0)}=\int_Mg_{\Phi_0}(\alpha,\Phi_0)vol_{g_{\Phi_0}}$.
\begin{lem}
Let $\Phi_0\in\widetilde{\mathcal{M}}_{\Phi_{qK}}(M)$ and $k\ge 2n+1$. Then 
there is an open neighborhood $W_{k,\Phi_0}\subset \mathcal{A}_k$ of 
$\pi_k(\Phi_0)$ which satisfies 
$W_{k,\Phi_0}\subset\pi_k\circ\varphi_{k,\Phi_0}(V_{k,\Phi_0}(\varepsilon))$ for any $\varepsilon >0$.
\end{lem}
\begin{proof}
Let the map 
\begin{eqnarray}
F:\varphi_{k,\Phi_0}(V_{k,\Phi_0}(\varepsilon))\times L^2_{k+1}({\rm 
Diff}_0(M))\times\mathbb{R}_{>0}\longrightarrow L^2_k(R_{\Phi_{qK}}(M))\nonumber
\end{eqnarray}
be given by $F(\Phi, f, c):=cf^*\Phi$ for 
$\Phi\in\varphi_{k,\Phi_0}(V_{k,\Phi_0}(\varepsilon))$, $f\in L^2_{k+1}({\rm Diff}_0(M))$ and 
$c\in\mathbb{R}_{>0}$.
We take
\begin{eqnarray}
(\alpha,X,t)&\in &V_{k,\Phi_0}\oplus 
L^2_{k+1}(TM)\oplus\mathbb{R}\nonumber\\
&=& 
T_{(\Phi_0,Id_M,1)}\{\varphi_{k,\Phi_0}(V_{k,\Phi_0}(\varepsilon))\times L^2_{k+1}({\rm Diff}_0(M))\times\mathbb{R}_{>0}\}\nonumber
\end{eqnarray}
where $V_{k,\Phi_0}:=\{\alpha\in L^2_k(E^1_{\Phi_0});\ d^*_0\alpha=0,\ 
<\alpha,\Phi_0>_{L^2(\Phi_0)}=0\}$. Then the differential of $F$ at 
$(\Phi_0,Id_M,1)$ is given by
\begin{eqnarray}
F_*|_{(\Phi_0,Id_M,1)}(\alpha,X,t)=t\Phi_0 + d\iota_X\Phi_0 + 
\alpha.\nonumber
\end{eqnarray}
Thus $F_*|_{(\Phi_0,Id_M,1)}:V_{k,\Phi_0}\oplus 
L^2_{k+1}(TM)\oplus\mathbb{R}\to L^2_k(E^1_{\Phi_0})$ is an isomorphism by Lemma 3.3, so there are 
some $\delta >0$ and neighborhood $N_{(Id_M,1)}\subset L^2_{k+1}({\rm 
Diff}_0(M))\times\mathbb{R}_{>0}=\mathcal{G}_{k+1}$ of $(Id_M,1)$ such that 
$F|_{\varphi_{k,\Phi_0}(V_{k,\Phi_0}(\delta))\times N_{(Id_M,1)}}$ is a 
diffeomorphism. In particular, $F(\varphi_{k,\Phi_0}(V_{k,\Phi_0}(\delta))\times 
N_{(Id_M,1)})$ is an open set of $L^2_k(R_{\Phi_{qK}}(M))$. Hence by putting 
\begin{eqnarray}
W_{k,\Phi_0}:=\pi_k\circ F(\varphi_{k,\Phi_0}(V_{k,\Phi_0}(\delta '))\times 
N_{(Id_M,1)})\nonumber
\end{eqnarray}
for $\delta  '=min\{\delta ,\varepsilon\}$, we have
\begin{eqnarray}
W_{k,\Phi_0}\subset\pi_k\circ 
F(\varphi_{k,\Phi_0}(V_{k,\Phi_0}(\varepsilon))\times 
N_{(Id_M,1)})=\pi_k\circ\varphi_{k,\Phi_0}(V_{k,\Phi_0}(\varepsilon)).\nonumber
\end{eqnarray}
\end{proof}
\begin{prop}
Let $\Phi_0\in\widetilde{\mathcal{M}}_{qK}$ and $n\ge 3$. Then there is 
$\varepsilon >0$ which satisfies the following condition. If 
$\Phi\in\varphi_{2n+1,\Phi_0}(V_{2n+1,\Phi_0}(\varepsilon))$ satisfies $d\Phi=0$, then 
$\Phi=\Phi_0$.
\end{prop}
\begin{proof}
Fix $\varepsilon >0$ and 
$\Phi\in\varphi_{2n+1,\Phi_0}(V_{2n+1,\Phi_0}(\varepsilon))$. Then we may write $\Phi=\rho(e^a)\Phi_0$ for $a\in 
L^2_{2n+1}(\hat{P}^1_{Q_{\Phi_0}})$ which satisfies $\bar{A}^1_{\Phi_0}(a)\in 
V_{2n+1,\Phi_0}(\varepsilon)$. If we put
\begin{eqnarray}
\rho_*(a)\beta:=\frac{d}{dt}\bigg|_{t=0}(e^{-ta})^*\beta\nonumber
\end{eqnarray}
for $\beta\in\Omega^{\cdot}(M)$, then it follows 
$\rho_*(a)\Phi_0=-\bar{A}^1_{\Phi_0}(a)$ and
\begin{eqnarray}
\rho 
(e^a)\Phi_0=\sum^{\infty}_{j=0}\frac{1}{j!}\{\rho_*(a)\}^j\Phi_0.\nonumber
\end{eqnarray}
Then
\begin{eqnarray}
d\Phi &=& d\sum^{\infty}_{j=0}\frac{1}{j!}\{\rho_*(a)\}^j\Phi_0\nonumber\\
 &=& -d\bar{A}^1_{\Phi_0}(a) + 
\sum^{\infty}_{j=2}\frac{1}{j!}d\{\rho_*(a)\}^j\Phi_0.
\end{eqnarray}
Since the linear operators
\begin{eqnarray}
\rho_* &:& L^2_{2n+1}(End(TM))\longrightarrow 
L^2_{2n+1}(End(\Lambda^4T^*M)),\nonumber\\
d &:& L^2_{2n+1}(\Lambda^4T^*M)\longrightarrow 
L^2_{2n}(\Lambda^5T^*M)\nonumber
\end{eqnarray}
are the bounded operators, there are constants $s,K_0,K_1>0$ depending only on 
$M$ and $\Phi_0$, such that
\begin{eqnarray}
\|\{\rho_*(a)\}^j\Phi_0\|_{L^2_{2n+1}} &\le & K_0s^j\| 
a\|^j_{L^2_{2n+1}}\|\Phi_0\|_{L^2_{2n+1}},\nonumber\\
\| d\beta\|_{L^2_{2n}} &\le & K_1\| \beta\|_{L^2_{2n+1}}\nonumber
\end{eqnarray}
for $a\in L^2_{2n+1}(\hat{P}^1_{Q_{\Phi_0}})$ and $\beta\in 
L^2_{2n+1}(\Lambda^4T^*M)$. So we have
\begin{eqnarray}
\|\sum^{\infty}_{j=2}\frac{1}{j!}d\{\rho_*(a)\}^j\Phi_0\|_{L^2_{2n}} &\le & 
\sum^{\infty}_{j=2}\frac{1}{j!}\|d\{\rho_*(a)\}^j\Phi_0\|_{L^2_{2n}}\nonumber\\
 &\le & K_0K_1\sum^{\infty}_{j=2}\frac{1}{j!}(s\| 
a\|_{L^2_{2n+1}})^j\|\Phi_0\|_{L^2_{2n+1}}\nonumber\\
 &=& K_0K_1s^2\|\Phi_0\|_{L^2_{2n+1}}\| a\|^2_{L^2_{2n+1}}f(s\| 
a\|_{L^2_{2n+1}}),\nonumber
\end{eqnarray}
where a $C^{\infty}$-function $f:\mathbb{R}\to\mathbb{R}$ is given by
\[ f(x)=
\left\{
\begin{array}{cc}
(e^x-1-x)/x^2 & (x\neq 0), \\
1/2 & (x = 0).
\end{array}
\right.
\]
If we take $\varepsilon\le (K_2 s)^{-1}$ where $K_2$ is the operator norm 
of the bounded operator
\begin{eqnarray}
(\bar{A}^1_{\Phi_0})^{-1}:L^2_{2n+1}(E^1_{\Phi_0})\longrightarrow 
L^2_{2n+1}(\hat{P}^1_{Q_{\Phi_0}}),\nonumber
\end{eqnarray}
then we have
\begin{eqnarray}
f(s\| a\|_{L^2_{2n+1}})\le f_{max}:=max\{f(x);\ x\in 
[0,1]\}<\infty.\nonumber
\end{eqnarray}
From now on, we suppose $d\Phi=0$. Then (2) gives 
\begin{eqnarray}
d\bar{A}^1_{\Phi_0}(a) = 
\sum^{\infty}_{j=2}\frac{1}{j!}d\{\rho_*(a)\}^j\Phi_0,\nonumber
\end{eqnarray}
so it follows
\begin{eqnarray}
\| d\bar{A}^1_{\Phi_0}(a)\|_{{L^2_{2n}}} &=& \| 
\sum^{\infty}_{j=2}\frac{1}{j!}d\{\rho_*(a)\}^j\Phi_0\|_{{L^2_{2n}}}\nonumber\\
 &\le & K_0K_1f_{max}s^2\|\Phi_0\|_{L^2_{2n+1}}\| a\|^2_{L^2_{2n+1}}.
\end{eqnarray}
\quad From Lemma 3.3 and Lemma 3.4, there is the decomposition
\begin{eqnarray}
L^2_{2n+1}(E^1_{\Phi_0}) &=& \mathbb{R}\Phi_0\oplus 
dL^2_{2n+2}(E^0_{\Phi_0})\oplus  d^*_1L^2_{2n+2}(E^2_{\Phi_0})\nonumber\\
 &=& \mathbb{R}\Phi_0\oplus L^2_{2n+1}(Im(\triangle_{\sharp})).\nonumber
\end{eqnarray}
Then $\bar{A}^1_{\Phi_0}(a)$ is an element of 
$d^*_1L^2_{2n+2}(E^2_{\Phi_0})$ since $\bar{A}^1_{\Phi_0}(a)\in V_{2n+1,\Phi_0}(\varepsilon)$. From 
Lemma 3.3, there is the Green operator
\begin{eqnarray}
G_{\sharp}: L^2_{2n-1}(Im(\triangle_{\sharp}))\longrightarrow 
L^2_{2n+1}(Im(\triangle_{\sharp})),\nonumber
\end{eqnarray}
which is the inverse operator of 
$\triangle_{\sharp}|_{L^2_{2n+1}(Im(\triangle_{\sharp}))}$. Since $d^*_0\bar{A}^1_{\Phi_0}(a)=0$, we have
\begin{eqnarray}
\| \bar{A}^1_{\Phi_0}(a)\|_{L^2_{2n+1}} &=& \| 
G_{\sharp}\triangle_{\sharp}\bar{A}^1_{\Phi_0}(a)\|_{L^2_{2n+1}}\nonumber\\
 &\le & K_3K_4\| d\bar{A}^1_{\Phi_0}(a)\|_{L^2_{2n}},
\end{eqnarray}
where $K_3$ and $K_4$ are the operator norms of $G_{\sharp}$ and $d^*_1$, 
respectively.\\
\quad Thus we have an inequality
\begin{eqnarray}
\| a \|_{L^2_{2n+1}} \le K_0K_1K_2K_3K_4s^2f_{max}\|\Phi_0\|_{L^2_{2n+1}}\| 
a\|^2_{L^2_{2n+1}},\nonumber
\end{eqnarray}
from (3) and (4). Then we obtain an estimate
\begin{eqnarray}
\| a \|_{L^2_{2n+1}}(1-K_0K_1K_2K_3K_4s^2f_{max}\|\Phi_0\|_{L^2_{2n+1}}\| 
a\|_{L^2_{2n+1}}) \le 0.\nonumber
\end{eqnarray}
So if we put 
$\varepsilon=min\{(K_2s)^{-1},(2K_0K_1K_2K_3K_4s^2f_{max}\|\Phi_0\|_{L^2_{2n+1}})^{-1}\}>0$, then it follows $\| a \|_{L^2_{2n+1}}=0$, which 
means $\Phi=\Phi_0$.
\end{proof}
\begin{prop}
Let $\Phi_0\in\widetilde{\mathcal{M}}_{qK}$ and $n\ge 3$. Then the set $\{\pi_k(\Phi_0)\}$ is an open subset of $\pi_k(\widetilde{\mathcal{M}}_{qK})$ for $k\ge 2n+1$.
\end{prop}
\begin{proof}
Take $\varepsilon >0$ as in Proposition 3.6. We can take an open set 
$W_{k,\Phi_0}$ of $\mathcal{A}_k$ such that $W_{k,\Phi_0}\subset\pi_k\circ\varphi_{k,\Phi_0}(V_{k,\Phi_0}(\varepsilon))$ from Lemma 3.5. Then 
$W_{k,\Phi_0}\cap\pi_k(\widetilde{\mathcal{M}}_{qK})$ is an open set of $\pi_k(\widetilde{\mathcal{M}}_{qK})$. Let $x$ be an element of $W_{k,\Phi_0}\cap\pi_k(\widetilde{\mathcal{M}}_{qK})$. Then we may write $x=\pi_k(\Phi)$ for some 
$\Phi\in\varphi_{k,\Phi_0}(V_{k,\Phi_0}(\varepsilon))$, but Proposition 3.6 gives $\Phi=\Phi_0$ since $\varphi_{k,\Phi_0}(V_{k,\Phi_0}(\varepsilon))\subset \varphi_{2n+1,\Phi_0}(V_{2n+1,\Phi_0}(\varepsilon))$. Hence we have $W_{k,\Phi_0}\cap\pi_k(\widetilde{\mathcal{M}}_{qK})=\{\pi_k(\Phi_0)\}$.
\end{proof}
Thus we have shown that $\pi_k(\widetilde{\mathcal{M}}_{qK})$ is a discrete 
subset of $\mathcal{A}^k$ for $k\ge 2n+1$ and $n\ge 3$. Next we are going 
to prove Theorem 1.3.
\begin{lem}
Let $\Phi_A,\Phi_B\in\widetilde{\mathcal{M}}_{\Phi_{qK}}(M)$ and $N\ge 
2n+1$. Suppose that $\pi_k(\Phi_A)=\pi_k(\Phi_B)$ for any $k\ge N$. Then there 
is $(f,c)\in\mathcal{G}={\rm Diff}_0(M)\times\mathbb{R}_{>0}$ such that 
$cf^*\Phi_A=\Phi_B$.
\end{lem}
\begin{proof}
Suppose that we have $\pi_k(\Phi_A)=\pi_k(\Phi_B)$ for any $k\ge N$. Then there 
is $(f_k,c_k)\in\mathcal{G}_{k+1}$ such that $c_kf^*_k\Phi_A=\Phi_B$ for 
each $k$, so we have $c_Nf^*_N\Phi_A=c_kf^*_k\Phi_A$. Since each $f_k$ is 
homotopic to $Id_M$, we have $[c_kf^*_k\Phi_A]_{dR}=[c_k\Phi_A]_{dR}$ where 
$[\theta]_{dR}\in H^4_{dR}(M)$ is the de Rham class of a closed form 
$\theta\in\Omega^4$. Then if we denote by $\Pi_h$ the harmonic projection with 
respect to $g_{\Phi_A}$, we have
\begin{eqnarray}
\Pi_h(c_kf^*_k\Phi_A)=\Pi_h(c_k\Phi_A)=c_k\Phi_A.\nonumber
\end{eqnarray}
Hence it follows that $c_k\Phi_A = c_N\Phi_A$, which gives $c_k = c_N$. 
Then $f_N\circ f_k^{-1}\in L^2_{N+1}({\rm Diff}_0(M))$ is an element of
\begin{eqnarray}
I_{\Phi_A}:=\{f\in L^2_{N+1}({\rm Diff}_0(M));\ 
f^*\Phi_A=\Phi_A\}.\nonumber
\end{eqnarray}
According to \cite{P}, if a $C^1$ diffeomorphism $f:M\to M$ preserves a smooth 
Riemannian metric, then $f$ is smooth. Since each element of $I_{\Phi_A}$ 
preserves smooth Riemannian metric $g_{\Phi_A}$, then $I_{\Phi_A}$ is a 
subgroup of ${\rm Diff}_0(M)$. So $\tilde{f}_k:=f_N\circ f_k^{-1}$ is an element of 
${\rm Diff}_0(M)$, then we may write $f_N=\tilde{f}_k\circ f_k$, which is an element of $L^2_{k+1}({\rm Diff}_0(M))$. Thus $f_N$ is an element of
\begin{eqnarray}
\bigcap^{\infty}_{k=N}L^2_{k+1}({\rm Diff}_0(M))={\rm Diff}_0(M).\nonumber
\end{eqnarray}
\end{proof}
Now we have a quotient space 
$\mathcal{M}_{qK}:=\widetilde{\mathcal{M}}_{qK}/\mathcal{G}$ and the quotient map
\begin{eqnarray}
\pi_{qK}:\widetilde{\mathcal{M}}_{qK}\longrightarrow\mathcal{M}_{qK}.\nonumber
\end{eqnarray}
\begin{defi}
{\rm Quaternionic K\"ahler structures 
$\{\Phi_t\}_{t\in\mathbb{R}}\subset\widetilde{\mathcal{M}}_{qK}$ is a continuous family if the map
\begin{eqnarray}
\widetilde{\Phi}_k:\mathbb{R}\longrightarrow 
L^2_k(R_{\Phi_{qK}}(M))\nonumber
\end{eqnarray}
defined by $\widetilde{\Phi}_k(t):=\Phi_t$ is a continuous map for each 
$k\ge 2n+1$.}
\end{defi}
Then a rigidity theorem for quaternionic K\"ahler structures is obtained as 
follows.
\begin{thm}
Let $\{\Phi_t\}_{t\in\mathbb{R}}\subset\widetilde{\mathcal{M}}_{qK}$ be a 
continuous family on a compact manifold $M$ of dimension $4n$ for $n\ge 3$. 
Then we have $\pi_{qK}(\Phi_t)=\pi_{qK}(\Phi_0)$ for any $t\in\mathbb{R}$.
\end{thm}
\begin{proof}
Since the maps $\pi_k\circ\widetilde{\Phi}_k$ are continuous maps, we have $\pi_k(\Phi_t)=\pi_k(\Phi_0)$ from Proposition 3.7. Then by Lemma 3.8, we obtain 
$\pi_{qK}(\Phi_t)=\pi_{qK}(\Phi_0)$.
\end{proof}
\section{Deformation complexes of torsion-free $G$-structures}
The purpose of this section is to give the proof of Lemma 3.3. We introduce 
new complexes; the deformation complexes of torsion-free $G$-structures. 
The new complex is elliptic at $k=1$ if $G$ satisfies a certain condition. 
Then we can construct an isomorphism between the deformation complex of 
torsion-free $Sp(n)Sp(1)$-structures and Goto's complex (1) of quaternionic 
K\"ahler structures.\\
\quad Let $G$ be a Lie subgroup of $O(N)$ and $M$ be a compact manifold of 
dimension $N$. Fix a torsion-free $G$-structure $Q\in \Gamma (R_G(M))$. Let 
$d^{\nabla}_k : \Omega^k(TM)\to \Omega^{k+1}(TM)$ be the covariant exterior 
derivative of the Levi-Civita connection $\nabla$ of $g_{Q}$. Then it is 
easy to see that $d^{\nabla}_k \Gamma(\hat{\textswab{g}}^k_Q)$ is a subspace 
of $\Gamma(\hat{\textswab{g}}^{k+1}_Q)$ for $k=1,2,\cdots$. Moreover, there 
is a following property.
\begin{prop}
Let $Q\in \Gamma (R_G(M))$ be a torsion-free $G$-structure, and $\nabla$ be 
the Levi-Civita connection of $g_Q$. Then $d^{\nabla}_{k+1}\circ 
d^{\nabla}_k(\Gamma (\Lambda^kT^*M\otimes TM))$ is a subspace of 
$\Gamma(\hat{\textswab{g}}^{k+2}_Q)$.
\end{prop}
\begin{proof}
First we suppose $k=0$. We are going to show $(d^{\nabla})^2X \in\Gamma 
(\hat{\textswab{g}}^2_Q)$ for each $X\in\mathcal{X}(M)$.\\
\quad Fix $p\in M$, and take a neighborhood $U$ of $p$ and local 
orthonormal frame $\xi_1,\cdots ,\xi_N \in\mathcal{X}(U)$ and its dual frame 
$\xi^1,\cdots ,\xi^N \in\Omega^1(U)$ as in the proof of Proposition 2.3. Let $R\in\Omega^2(\hat{\textswab{g}}_Q)$ be 
a curvature tensor of $\nabla$, and we may write 
$R(\xi_i,\xi_j)\xi_l=\sum_mR^m_{ijl}\xi_m$ and $X=\sum_lX^l\xi_l\in\mathcal{X}(U)$ on $U$. Since the curvature tensor of $G$-connection is a section of the vector bundle induced by the adjoint action of $\textswab{g}$, so $\sum_{l,m}R^m_{ijl}\xi^l\otimes\xi_m\in\Gamma(T^*U\otimes TU)$ is a section of $\hat{\textswab{g}}^1_Q|_U$. Then we have
\begin{eqnarray}
(d^{\nabla})^2X = 
\sum_{i,j,l,m}R^m_{ijl}X^l\xi^i\wedge\xi^j\otimes\xi_m.\nonumber
\end{eqnarray}
So it follows that
\begin{eqnarray}
(d^{\nabla})^2X &=& 
\sum_{i,j,l,m}(-R^m_{jli}-R^m_{lij})X^l\xi^i\wedge\xi^j\otimes\xi_m\nonumber\\
 &=& 
2\sum_{j,l}X^l\xi^j\wedge(\sum_{i,m}R^m_{jli}\xi^i\otimes\xi_m)\in\Gamma (\hat{\textswab{g}}^2_Q).\nonumber
\end{eqnarray}
from the first Bianchi identity, $R^m_{ijl} + R^m_{jli} + R^m_{lij} = 0$.
\quad Next we show the case of $k\ge 1$. It suffices to show that 
$(d^{\nabla})^2(\alpha\otimes X)\in\Gamma(\hat{\textswab{g}}^{k+2}_Q)$ for each 
$\alpha\in\Omega^k(M)$ and $X\in\mathcal{X}(M)$. Then we have
\begin{eqnarray}
(d^{\nabla})^2(\alpha\otimes X) &=& d^{\nabla}(d\alpha\otimes X + 
(-1)^k\alpha\wedge\nabla X)\nonumber\\
&=& d^2\alpha\otimes X + (-1)^{k+1}d\alpha\wedge\nabla X + 
(-1)^kd\alpha\wedge\nabla X \nonumber\\
&\ & + \alpha\wedge (d^{\nabla})^2X \nonumber\\
&=& \alpha\wedge (d^{\nabla})^2X 
\in\Gamma(\hat{\textswab{g}}^{k+2}_Q)\nonumber
\end{eqnarray}
from the case of $k=0$.
\end{proof}
We define a differential operator $d^{Q}_k:\Gamma (\hat{P}^k_Q)\to\Gamma 
(\hat{P}^{k+1}_Q)$ by $d^{Q}_k:=pr_{\hat{P}^{k+1}_Q}\circ d^{\nabla}_k$, 
where $pr_{\hat{P}^k_Q}:\Lambda^kT^*M\otimes TM\to\hat{P}^k_Q$ is the 
orthogonal projection. Then from Proposition 4.1, we obtain the deformation complex 
of torsion-free $G$-structures
\begin{eqnarray}
\cdots \stackrel{d^Q_{k-1}}{\longrightarrow}\Gamma 
(\hat{P}^k_Q)\stackrel{d^Q_k}{\longrightarrow}\Gamma 
(\hat{P}^{k+1}_Q)\stackrel{d^Q_{k+1}}{\longrightarrow}\cdots.
\end{eqnarray}
\quad Next we will see that the complex (5) is elliptic at $k=1$. We denote by $Sb_k(u)$ the symbol of the differential operator 
$d_k^Q$ at $u\in V^*-\{ 0\}$. Let $pr_{P^k_{\textswab{g}}}:\Lambda^k\otimes 
V\to P^k_\textswab{g}$ be the orthogonal projection. Then we have
\begin{eqnarray}
Sb_k(u)(X) &=& pr_{P^{k+1}_{\textswab{g}}}(u\wedge X)\nonumber\\
 &=& pr_{P^{k+1}_{\textswab{g}}}\bigl(\sum_{i_1,\cdots,i_k,j}X_{i_1\cdots 
i_k}^j(u\wedge v^{i_1}\wedge\cdots\wedge v^{i_k})\otimes v_j\bigr)\nonumber
\end{eqnarray}
for $X=\sum_{i_1,\cdots,i_k,j}X_{i_1\cdots i_k}^j v^{i_1}\wedge\cdots\wedge 
v^{i_k}\otimes v_j\in P^k_\textswab{g}$. To prove that the complex (5) is 
elliptic at $k=1$, we have to see the complex
\begin{eqnarray}
\cdots 
\stackrel{Sb_{k-1}(u)}{\longrightarrow}P^k_\textswab{g}\stackrel{Sb_k(u)}{\longrightarrow}P^{k+1}_\textswab{g}\stackrel{Sb_{k+1}(u)}{\longrightarrow}\cdots\nonumber
\end{eqnarray}
is the exact sequence at $P^1_\textswab{g}$.\\
\quad We have an orthogonal decomposition $V=\mathbb{R}v\oplus W_v$ with 
respect to $g_0$ for each $v\in V-\{ 0\}$. The decomposition induces the 
orthogonal projection $p_v:End(V)\to End(W_v)$, then we can consider the following 
conditions for $G\subset O(N)$.\\
\\
($\mathbf{C1}$)\ The linear map $p_v|_{\textswab{g}}:\textswab{g}\to 
End(W_v)$ is injective for each $v\in V-\{ 0\}$.\\
\\
The condition ($\mathbf{C1}$) is equivalent to the following condition.\\
\\
($\mathbf{C2}$)\ Let $v_1 , v_2 , \cdots, v_N\in V$ be any orthonormal 
basis and $v^1 , v^2 , \cdots, v^N\in V^*$ be its dual basis. Then for all 
$A=A_i^jv^i\otimes v_j\in\textswab{g}$, $A=0$ if $A_i^j=0$ for $i,j\neq 1$.\\
\begin{lem}
Suppose that the Lie subgroup $G\subset O(N)$ satisffies ($\mathbf{C1}$). 
For all $a\in End(V)$ and $u\in V^*-\{ 0\}$, we may write $a=b+u\otimes w$ 
for some $b\in\textswab{g}$ and $w\in V$ if $u \wedge a\in\textswab{g}^2$.
\end{lem}
\begin{proof}
We may take $u\in V^*-\{ 0\}$ as $g_0(u,u)=1$. Then we fix an orthonormal 
basis $v^1=u , v^2 , \cdots, v^N\in V^*$ and its dual basis $v_1 , v_2 , 
\cdots, v_N\in V$with respect to $g_0$. Suppose that $v^1 \wedge 
\sum_{j,k}a_j^k v^j\otimes v_k = \sum_{i,j,k}B_{ij}^k v^i\wedge v^j\otimes v_k$ for $a = 
\sum_{j,k}a_j^k v^j\otimes v_k\in End(V)$ and $\sum_{j,k}B_{ij}^k v^j\otimes v_k 
\in\textswab{g}$ for $i = 1,2,\cdots,N$. Then
\begin{eqnarray}
v^1 \wedge \sum_{j,k}a_j^k v^j\otimes v_k = \sum_{i\textless 
j}\sum_k(B_{ij}^k-B_{ji}^k) v^i\wedge v^j\otimes v_k.\nonumber
\end{eqnarray}
So we have
\begin{eqnarray}
a_j^k &=& B_{1j}^k-B_{j1}^k \quad (j\neq 1),\\
B_{ij}^k-B_{ji}^k &=& 0 \quad (i,j\neq 1).
\end{eqnarray}
Now $B_{ij}^k$ satisfies $B_{ij}^k=-B_{ik}^j$ for any $i,j,k$ since 
$\sum_{j,k}B_{ij}^k v^j\otimes v_k \in\textswab{g}\subset 
\mathbf{so}(N)$. Then (7) gives $B_{ij}^k=0$ for $i,j,k\neq 1$. Hence $B_{ij}^k=0$ for 
$i\neq 1$ and $j,k=1,\cdots, N$ from ($\mathbf{C2}$). Then
\begin{eqnarray}
a &=& \sum_{j,k}a_j^k v^j\otimes v_k\nonumber\\
&=& \sum_{k}a_1^k v^1\otimes v_k + \sum_{j\neq 1}\sum_k(B_{1j}^k-B_{j1}^k) 
v^j\otimes v_k.\nonumber
\end{eqnarray}
Since $B_{j1}^k=0$ for $j\neq 1$ and $k=1,\cdots, N$, we have
\begin{eqnarray}
a &=& v^1\otimes \sum_{k}a_1^k v_k + \sum_{j,k}B_{1j}^k v^j\otimes v_k - 
\sum_{k}B_{11}^k v^1\otimes v_k.\nonumber
\end{eqnarray}
Then we have finished the proof by putting $b=\sum_{j,k}B_{1j}^k v^j\otimes 
v_k$ and $w=\sum_{k}(a_1^k v_k - B_{11}^k) v_k$.
\end{proof}
\begin{prop}
Let $G$ be a Lie subgroup of $O(N)$ satisfying the condition 
($\mathbf{C1}$). Then the complex (5) is elliptic at $k=1$ for any torsion-free 
$G$-structure $Q$.
\end{prop}
\begin{proof}
Let $a\in P^1_\textswab{g}$, $u\in V^*-\{0\}$ and $Sb_1(u)a=0$. Since 
$Sb_1(u)a=0$ means that $u\wedge a$ is an element of 
$\textswab{g}^2$, we may write $a=b+u\wedge w$ for some $b\in\textswab{g}$ and $w\in V$ 
from Lemma 4.2. Hence we obtain
\begin{eqnarray}
a = pr_{P^1_{\textswab{g}}}(a) = pr_{P^1_{\textswab{g}}}(b+u\wedge w) = 
pr_{P^1_{\textswab{g}}}(u\wedge w) = Sb_0(u)w.\nonumber
\end{eqnarray}
\end{proof}
\begin{prop}
If a Lie group $G\subset O(N)$ is defined by
\begin{eqnarray}
G:=\{g\in GL_N\mathbb{R};\rho (g)\Phi^V=\Phi^V\}\nonumber
\end{eqnarray}
for $\Phi^V\in\bigoplus^l_{i=1}\Lambda^{p_i}$, then $G$ satisfies the 
condition ($\mathbf{C1}$).
\end{prop}
\begin{proof}
Let $v_1,\cdots,v_n$ be an orthonormal basis of $V=\mathbb{R}^n$, and 
$v^1,\cdots,v^n$ be its dual basis. We suppose that $A=\sum_{i,j}A_i^jv^i\otimes 
v_j$ is an element of $\textswab{g}$ and $A_i^j=0$ for $i,j\neq 1$. From 
the definition of $G$, we have
\begin{eqnarray}
\sum_{1\le j\le n}A_1^jv^1\wedge\iota_{v_j}\Phi^V + \sum_{2\le i\le 
n}A_i^1v^i\wedge\iota_{v_1}\Phi^V = 0.\nonumber
\end{eqnarray}
So we have
\begin{eqnarray}
\sum_{1\le j\le n}A_1^jv^1\wedge\iota_{v_j}\Phi^V &=& 0,\nonumber\\
\sum_{2\le i\le n}A_i^1v^i\wedge\iota_{v_1}\Phi^V &=& 0.\nonumber
\end{eqnarray}
Thus we have shown $\sum_{1\le j\le n}A_1^jv^1\otimes v_j$ and $\sum_{2\le 
i\le n}A_i^1v^i\otimes v_1$ are the elements of 
$\textswab{g}\subset\mathbf{so}(N)$. Hence we obtain $A_1^j=A_i^1=0$ for any $i,j$.
\end{proof}
It is easy to see 
$A^{k+1}_{\Phi_0}(d^{\nabla}_k\beta)=dA^k_{\Phi_0}(\beta)$ for $\beta\in\Omega^k(TM)$ by direct calculation using local orthonormal 
frame appear in the proof of Proposition 2.3, and we have the following 
proposition.
\begin{prop}
Let $\Phi_0\in \widetilde{\mathcal{M}}_{\Phi^V}(M)$. Suppose that $dim 
E^k_{\Phi^V}=dim P^k_{\textswab{g}}$ for $k=l,l+1$. Then $\bar{A}^l$ and 
$\bar{A}^{l+1}$ are isomorphisms and 
$d\circ\bar{A}^l_{Q_{\Phi_0}}=\bar{A}^{l+1}_{\Phi_0}\circ d^{Q_{\Phi_0}}_{l+1}$.
\end{prop}
Now, we have $dim E^1_{\Phi^V}=dim P^1_{\textswab{g}}$ by the definition of 
$G$. So we have the following.
\begin{prop}
Let $\Phi_0\in \widetilde{\mathcal{M}}_{\Phi^V}(M)$ and suppose that $dim 
E^2_{\Phi^V}=dim P^2_{\textswab{g}}$. Then the complex (1) is an 
elliptic comlex at $k=1$.
\end{prop}
Since $dim E^2_{\Phi_{qK}} = dim P^2_{\mathbf{sp}(n)\oplus\mathbf{sp}(1)}$ 
from the proof of Theorem 3.2, we have shown Lemma 3.3.
\section{Bochner-Weitzenb\"ock formulas on the \\
quaternionic K\"ahler manifold}
\quad In this section we give a proof of Lemma 3.4.
\begin{lem}
Let $\Phi_0\in\widetilde{\mathcal{M}}_{qK}$, and 
$\triangle_{\sharp}=dd^*_0+d^*_1d$ be as in Lemma 3.3. Then we have $Ker\ 
\triangle_{\sharp}=\mathbb{R}\Phi_0$ for $n\ge 3$.
\end{lem}
\begin{proof}
By the definition of $E^k_{\Phi_{qK}}$, we have 
$E^1_{\Phi_{qK}}=A^1(End\mathbb{R}^{4n})$. Since there is a natural decomposition
\begin{eqnarray}
End\mathbb{R}^{4n}=\mathbf{so}(4n)\oplus\mathbb{R}(Id_{\mathbb{R}^{4n}})\oplus 
\mathbf{symm}_0(4n),\nonumber
\end{eqnarray}
where $\mathbf{symm}_0(4n)=\{A\in End\mathbb{R}^{4n};\ ^tA=A,\ 
trace(A)=0\}$, we have
\begin{eqnarray}
E^1_{\Phi_{qK}} &=& A^1(\mathbf{so}(4n))\oplus 
A^1(\mathbb{R}Id_{\mathbb{R}^{4n}})\oplus A^1(\mathbf{symm}_0(4n))\nonumber\\
&=& A^+\oplus \mathbb{R}\Phi_{qK}\oplus A^-,\nonumber
\end{eqnarray}
by putting $A^+=A^1(\mathbf{so}(4n))$, $A^-=A^1(\mathbf{symm}_0(4n))$. Then 
the above decomposition induces $E^1_{\Phi_0}=\hat{A}^+_{\Phi_0}\oplus 
\hat{\mathbb{R}}_{\Phi_0}\oplus \hat{A}^-_{\Phi_0}$. Note that 
$A^+\otimes\mathbb{C}\cong\lambda^2_0\sigma^2$, 
$A^-\otimes\mathbb{C}\cong\lambda^2_1\sigma^2\oplus\lambda^2_0$ and $\mathbb{C}\Phi_{qK}\cong\sigma^0$.\\
\quad Let the map $J:E^1_{\Phi_0}\to E^1_{\Phi_0}$ be given by 
$J(\alpha):=*(\alpha\wedge\Phi^{n-2}_0)$ where $*$ is the Hodge star operator with 
respect to $g_{\Phi_0}$. Then we can calculate $J(\alpha)$ by using 
decomposition $E^1_{\Phi_0}=\hat{A}^+_{\Phi_0}\oplus \hat{\mathbb{R}}_{\Phi_0}\oplus 
\hat{A}^-_{\Phi_0}$.\\
(i)\ Let $\alpha=\Phi_0$. Then we have
\begin{eqnarray}
J(\Phi_0)=*(\Phi_0\wedge\Phi^{n-2}_0)=*\Phi^{n-1}_0=\frac{|\Phi_{qK}|^2}{c_n}\Phi_0\nonumber
\end{eqnarray}
where $c_n$ is given by $\Phi_{qK}^n=c_nvol_{g_0}$.\\
(ii)\ Let 
$\alpha=\sum_{i,j}a_i^j\xi^i\wedge\iota_{\xi_j}\Phi_0\in\hat{A}^+_{\Phi_0}$, where $\xi^i,\xi_j$ are as in the proof of Proposition 2.3, and 
$a_i^j=-a_j^i$. Then we have
\begin{eqnarray}
J(\alpha) &=& 
*(\sum_{i,j}a_i^j\xi^i\wedge\iota_{\xi_j}\Phi_0\wedge\Phi^{n-2}_0)\nonumber\\
 &=& 
\frac{1}{n-1}*(\sum_{i,j}a_i^j\xi^i\wedge\iota_{\xi_j}\Phi^{n-1}_0)\nonumber\\
 &=& 
\frac{1}{n-1}\sum_{i,j}a_i^j\iota_{\xi_i}(\xi^j\wedge*\Phi^{n-1}_0)\nonumber\\
 &=& 
-\frac{1}{n-1}\sum_{i,j}a_i^j\xi^j\wedge\iota_{\xi_i}*\Phi^{n-1}_0\nonumber\\
 &=& 
\frac{1}{n-1}\frac{|\Phi_{qK}|^2}{c_n}\sum_{i,j}a_j^i\xi^j\wedge\iota_{\xi_i}\Phi_0 = \frac{1}{n-1}\frac{|\Phi_{qK}|^2}{c_n}\alpha.\nonumber
\end{eqnarray}
(iii)\ Let $\alpha\in\hat{A}^-_{\Phi_0}$. By calculating in the same way as 
(ii), we have
\begin{eqnarray}
J(\alpha) = -\frac{1}{n-1}\frac{|\Phi_{qK}|^2}{c_n}\alpha.\nonumber
\end{eqnarray}
From (i)(ii)(iii), it follows
\begin{eqnarray}
J(\alpha) = \frac{|\Phi_{qK}|^2}{c_n}(\alpha_0 + 
\frac{1}{n-1}\alpha_+-\frac{1}{n-1}\alpha_-)\nonumber
\end{eqnarray}
for $\alpha = \alpha_0 + \alpha_+ + 
\alpha_-\in\Gamma(\hat{\mathbb{R}}_{\Phi_0}\oplus \hat{A}^+_{\Phi_0}\oplus \hat{A}^-_{\Phi_0})$.\\
\quad Let $\alpha$ is an element of $Ker\ \triangle_{\sharp}$, which means 
$d\alpha=d^*_0\alpha=0$. Then
\begin{eqnarray}
d^*J(\alpha)=-*d(\alpha\wedge\Phi_0)=-*(d\alpha\wedge\Phi_0)=0.\nonumber
\end{eqnarray}
Let $p_0:\Lambda^3T^*M\to E^0_{\Phi_0}$ be the orthogonal projection. Since 
the formal adjoint $d^*_0$ is given by $d^*_0=p_0d^*$, we have two 
equations
\begin{eqnarray}
\frac{c_n}{|\Phi_{qK}|^2}d^*J(\alpha) &=& d^*(\alpha_0 + 
\frac{1}{n-1}\alpha_+-\frac{1}{n-1}\alpha_-)=0,\\
d^*_0\alpha &=& p_0d^*(\alpha_0 + \alpha_+ + \alpha_-)=0.
\end{eqnarray}
Then by calculating $(n-1)p_0\times (8)+(9)$, we obtain
\begin{eqnarray}
np_0d^*\alpha_0 + 2p_0d^*\alpha_+ =0.
\end{eqnarray}
Since $d^*\alpha_0\in\Gamma(E^0_{\Phi_0})$, the equation (10) is equivalent 
to
\begin{eqnarray}
nd^*\alpha_0 + 2p_0d^*\alpha_+ =0.
\end{eqnarray}
\quad There are non-trivial $Sp(n)Sp(1)$-equivariant maps
\begin{eqnarray}
\mathbf{T}_0:\mathbb{R}\Phi_{qK}\to\Lambda^0,\quad\mathbf{T}_1:E^0_{\Phi_{qK}}\to\Lambda^1,\quad\mathbf{T}_2:A^+\to\Lambda^2.\nonumber
\end{eqnarray}
Since $\mathbf{T}_0$ and $\mathbf{T}_1$ are isomorphisms and $\mathbf{T}_2$ is injective, each $\mathbf{T}_i$ is determmined uniquely up to scalar multiple by 
Schur's lemma. Then the bundle maps
\begin{eqnarray}
\hat{\mathbf{T}}_{0,\Phi_0} &:& 
\hat{\mathbb{R}}_{\Phi_0}\longrightarrow\Lambda^0T^*M,\nonumber\\
\hat{\mathbf{T}}_{1,\Phi_0} &:& 
E^0_{\Phi_0}\longrightarrow\Lambda^1T^*M,\nonumber\\
\hat{\mathbf{T}}_{2,\Phi_0} &:& 
\hat{A}^+_{\Phi_0}\longrightarrow\Lambda^2T^*M\nonumber
\end{eqnarray}
are induced by $\widetilde{Q_{\Phi_0}}$ and each $\mathbf{T}_i$. From 
Schur's lemma, there are nonzero constants $C_0,C_+\in\mathbb{R}$ which satisfy
\begin{eqnarray}
\hat{\mathbf{T}}_{1,\Phi_0}(d^*\alpha_0) &=& 
C_0d\hat{\mathbf{T}}_{0,\Phi_0}(\alpha_0),\nonumber\\
\hat{\mathbf{T}}_{1,\Phi_0}(p_0d^*\alpha_+) &=& 
C_+d^*\hat{\mathbf{T}}_{2,\Phi_0}(\alpha_+),\nonumber
\end{eqnarray}
since $\mathbb{R}\Phi_{qK}$, $E^0_{\Phi_{qK}}$ and $A^+$ are the 
irreducible $Sp(n)Sp(1)$-modules. So the equation (11) is equivalent to
\begin{eqnarray}
nC_0d\hat{\mathbf{T}}_{0,\Phi_0}(\alpha_0) + 
2C_+d^*\hat{\mathbf{T}}_{2,\Phi_0}(\alpha_+)=0,\nonumber
\end{eqnarray}
which gives 
$d\hat{\mathbf{T}}_{0,\Phi_0}(\alpha_0)=d^*\hat{\mathbf{T}}_{2,\Phi_0}(\alpha_+)=0$. Since $\hat{\mathbf{T}}_{1,\Phi_0}$ is an 
isomorphism, we obtain
\begin{eqnarray}
d^*\alpha_0=p_0d^*\alpha_+=0.
\end{eqnarray}
\quad Next we consider the condition $d\alpha=0$. There is the 
decomposition
\begin{eqnarray}
E^2_{\Phi_0} = (\hat{\lambda^3_0\sigma^3})_{\Phi_0}\oplus 
(\hat{\lambda^3_1\sigma^3})_{\Phi_0}\oplus (\hat{\lambda^3_0\sigma^1})_{\Phi_0}\oplus 
(\hat{\lambda^3_1\sigma^1})_{\Phi_0}\oplus 
(\hat{\lambda^1_0\sigma^3})_{\Phi_0}\oplus (\hat{\lambda^1_0\sigma^1})_{\Phi_0}\nonumber
\end{eqnarray}
induced by the decomposition of $E^2_{\Phi_{qK}}$ as in Section 3. By taking 
the orthogonal projection 
$\Pi_{\lambda^3_0\sigma^3}:E^2_{\Phi_0}\to(\hat{\lambda^3_0\sigma^3})_{\Phi_0}$, we have
\begin{eqnarray}
\Pi_{\lambda^3_0\sigma^3}(d\alpha_+)=0,
\end{eqnarray}
because the irreducible $Sp(n)Sp(1)$-decomposition of 
$V^*\otimes(\mathbb{R}\Phi_{qK}\oplus A^-)$ does not contain the component of 
$\lambda^3_0\sigma^3$.\\
\quad Since the space $A^+$ is isomorphic to $\lambda^2_0\sigma^2$ as an 
$Sp(n)Sp(1)$-module, there are the irreducible decomposition 
\begin{eqnarray}
T^*M\otimes\hat{A}^+_{\Phi_0} &=& 
(\hat{\lambda^3_1\sigma^3})_{\Phi_0}\oplus (\hat{\lambda^3_0\sigma^3})_{\Phi_0}\oplus 
(\hat{\lambda^1_0\sigma^3})_{\Phi_0}\nonumber\\
 &\ & \oplus (\hat{\lambda^3_0\sigma^1})_{\Phi_0}\oplus 
(\hat{\lambda^3_1\sigma^1})_{\Phi_0}\oplus (\hat{\lambda^1_0\sigma^1})_{\Phi_0},\nonumber
\end{eqnarray}
and orthogonal projections $pr_{\lambda^p_q\sigma^r}:T^*M\otimes\hat{A}^+_{\Phi_0}\to(\hat{\lambda^p_q\sigma^r})_{\Phi_0}$. Then differential operator $D_{a,b}$ on $\Gamma(\hat{A}^+_{\Phi_0})$ are defined by
\begin{eqnarray}
D_{1,1}:=pr_{\lambda^3_1\sigma^3}\circ\nabla,\quad 
D_{1,3}:=pr_{\lambda^3_0\sigma^3}\circ\nabla,\quad 
D_{1,-2}:=pr_{\lambda^1_0\sigma^3}\circ\nabla,\quad\nonumber\\
D_{-1,1}:=pr_{\lambda^3_1\sigma^1}\circ\nabla,\quad 
D_{-1,3}:=pr_{\lambda^3_0\sigma^1}\circ\nabla,\quad 
D_{-1,-2}:=pr_{\lambda^1_0\sigma^1}\circ\nabla,\nonumber
\end{eqnarray}
where $\nabla$ is the Levi-Civita connection of $g_{\Phi_0}$. According to 
\cite{Ho}, there are equations for $B_{a,b}:=(D_{a,b})^*D_{a,b}$ and the scalar 
curvature $\kappa_{g_{\Phi_0}}$, where $(D_{a,b})^*$ is the formal adjoint 
of $D_{a,b}$,
\begin{eqnarray}
\frac{1}{n+2}\kappa_{g_{\Phi_0}} &=& -B_{1,1} + 2B_{1,3} + 2nB_{1,-2}\\
&\ & - B_{-1,1} + 2B_{-1,3} + 2nB_{-1,-2},\nonumber\\
\frac{2}{n+2}\kappa_{g_{\Phi_0}} &=& -2(B_{1,1} + B_{1,3} + B_{1,-2})\\
&\ & + 4(B_{-1,1} + B_{-1,3} + B_{-1,-2}),\nonumber\\
\frac{8}{n+2}\kappa_{g_{\Phi_0}} &=& -2(n+2)B_{1,1} + 4(n-1)B_{1,3} - 
4n(n-1)B_{1,-2}\\
&\ & + 4(n+2)B_{-1,1} - 8(n-1)B_{-1,3} + 8n(n-1)B_{-1,-2}.\nonumber
\end{eqnarray}
Now we have $p_0d^*\alpha_+=0$ and $\Pi_{\lambda^3_0\sigma^3}(d\alpha_+)=0$ 
from (12)(13), which gives $B_{-1,-2}\alpha_+=0$ and $B_{1,3}\alpha_+=0$, 
respectively. So by calculating $2(n^2-n-2)\times (14) - n(n+3)\times (15) + 
\frac{2n+1}{2}\times (16)$, we have
\begin{eqnarray}
&\ & -(2n+1)(n-2)B_{1,1}\alpha_+ - 2(n-2)(n+2)B_{-1,1}\alpha_+\nonumber\\
&\ & \quad\quad\quad\quad - 4(2n+1)(n+1)B_{-1,3}\alpha_+\nonumber\\
&\ & = 0.\nonumber
\end{eqnarray}
Then by taking $L^2$ inner product with $\alpha_+$, it follows that
\begin{eqnarray}
&\ &-(2n+1)(n-2)\|D_{1,1}\alpha_+\|^2_{L^2} - 
2(n-2)(n+2)\|D_{-1,1}\alpha_+\|^2_{L^2}\nonumber\\
&\ & \quad\quad\quad\quad - 
4(2n+1)(n+1)\|D_{-1,3}\alpha_+\|^2_{L^2}\nonumber\\
&\ &  = 0,\nonumber
\end{eqnarray}
which gives $B_{1,1}\alpha_+ = B_{-1,1}\alpha_+ = B_{-1,3}\alpha_+ = 0$ 
since $n\ge 3$. Hence (14) and (15) gives $2nB_{1,-2}\alpha_+ = 
\frac{1}{n+2}\kappa_{g_{\Phi_0}}\alpha_+$ and $-B_{1,-2}\alpha_+ = 
\frac{1}{n+2}\kappa_{g_{\Phi_0}}\alpha_+$, respectively. Then by vanishing $B_{1,-2}$ from above 
two equations, we obtain $\frac{2n+1}{n+2}\kappa_{g_{\Phi_0}}\alpha_+=0$. 
Since we suppose the scalar curvature is nonzero, we have $\alpha_+ = 0$.\\
\quad From (11) and $\alpha_+ = 0$, we have $d^*\alpha_0=0$. If we write 
$\alpha_0=f\Phi_0$ for $f\in C^{\infty}(M)$, then $d^*\alpha_0=0$ means 
$df=0$ since the map $*(\cdot\wedge*\Phi_0):T^*M\to\Lambda^3 T^*M$ is injective. 
So $\alpha_0$ is given by $\alpha_0=c\Phi_0$ for $c\in\mathbb{R}$.\\
\quad Thus it follows $d\alpha_-=0$, $d^*\alpha_-=0$ from (8) and 
$d\alpha=d\alpha_0=d\alpha_+=0$. But there is no nonzero harmonic forms on 
$\Gamma((\hat{\lambda^2_1\sigma^2})_{\Phi_0}\oplus (\hat{\lambda^2_0})_{\Phi_0})$ 
according to the vanishing theorems on the quaternionic K\"ahler manifold 
\cite{Ho}\cite{S-W}, hence we obtain $\alpha_-=0$.
\end{proof}
\section{Quaternionic K\"ahler metrics and the reduced frame bundles}
\quad Let $g$ be a Riemannian metric on $M$ whose holonomy group $Hol(g)$ 
is isomorphic to $Sp(n)Sp(1)$. Then we set
\begin{eqnarray}
\widetilde{\mathcal{M}}_{qK}(g):=\{\Phi\in\widetilde{\mathcal{M}}_{qK};\ 
g_{\Phi}=g\}.\nonumber
\end{eqnarray}
The purpose of this section is showing that there is a unique element in 
$\widetilde{\mathcal{M}}_{qK}(g)$.\\
\quad We use following proposition, which can be seen in \cite{J} p.46.
\begin{prop}[{\cite{J}}]
Let $g$ be a quaternionic K\"ahler metric on a connected manifold $M$ of 
dimension $4n$. Then there is a one-to-one correspondence between 
$\widetilde{\mathcal{M}}_{qK}(g)$ and homogeneous space $Sp(n)Sp(1)\backslash N(Sp(n)Sp(1))$, where 
$N(Sp(n)Sp(1))$ is defined by
\begin{eqnarray}
N(Sp(n)Sp(1)):=\{x\in O(4n);\ x\{Sp(n)Sp(1)\}x^{-1}\subset 
Sp(n)Sp(1)\}.\nonumber
\end{eqnarray}
\end{prop}
\begin{prop}
Let $g$ be a quaternionic K\"ahler metric on a manifold $M$ of dimension 
$4n\ge 12$. Then $\widetilde{\mathcal{M}}_{qK}(g)$ has only one element.
\end{prop}
\begin{proof}
From Proposition 5.1, it suffices to show that
\begin{eqnarray}
N(Sp(n)Sp(1))=Sp(n)Sp(1).\nonumber
\end{eqnarray}
Let $\rho$ be as in Section 3. If we take $x\in N(Sp(n)Sp(1))$, then 
$xhx^{-1}$ is an element of $Sp(n)Sp(1)$ for any $h\in Sp(n)Sp(1)$. So it follows 
that
\begin{eqnarray}
\rho(xhx^{-1})\Phi_{qK} &=& \Phi_{qK},\nonumber\\
\rho(h)\rho(x^{-1})\Phi_{qK} &=& \rho(x^{-1})\Phi_{qK},\nonumber
\end{eqnarray}
for any $h\in Sp(n)Sp(1)$. It means that $\rho(x^{-1})\Phi_{qK}$ is an 
element of
\begin{eqnarray}
(\Lambda^4)^{Sp(n)Sp(1)}:=\{\alpha\in\Lambda^4;\ \rho(h)\alpha = \alpha\ 
{\rm for\ any}\ h\in Sp(n)Sp(1)\}.\nonumber
\end{eqnarray}
Then we may write $\rho(x^{-1})\Phi_{qK}=\lambda\Phi_{qK}$ for 
$\lambda\in\mathbb{R}$ since $(\Lambda^4)^{Sp(n)Sp(1)}\cong\mathbb{R}$ from the 
irreducible $Sp(n)Sp(1)$-decomposition of $\Lambda^4$ in Section 3. Since we have
\begin{eqnarray}
(\rho(x^{-1})\Phi_{qK})^n = \rho(x^{-1})\Phi_{qK}^n = 
det(x)\Phi_{qK}^n,\nonumber
\end{eqnarray}
it follows that $\lambda=\pm 1$. So we have
\begin{eqnarray}
N(Sp(n)Sp(1))=\{x\in O(4n);\ \rho(x^{-1})\Phi_{qK}=\pm\Phi_{qK}\},\nonumber
\end{eqnarray}
and it follows that the irreducible $Sp(n)Sp(1)$-module 
$\mathbb{R}\omega_I\oplus\mathbb{R}\omega_J\oplus\mathbb{R}\omega_K\subset\Lambda^2$ is an 
irreducible $N(Sp(n)Sp(1))$-module. Then we may write
\begin{eqnarray}
\rho(x^{-1})\omega_I &=& A_1\omega_I + A_2\omega_J + 
A_3\omega_K,\nonumber\\
\rho(x^{-1})\omega_J &=& B_1\omega_I + B_2\omega_J + 
B_3\omega_K,\nonumber\\
\rho(x^{-1})\omega_K &=& C_1\omega_I + C_2\omega_J + C_3\omega_K,\nonumber
\end{eqnarray}
for some $A_i,B_i,C_i\in\mathbb{R}$. If $\rho(x^{-1})\Phi_{qK}=-\Phi_{qK}$, 
then $A_1^2 + A_2^2 + A_3^2 = -1$. Hence $\rho(x^{-1})\Phi_{qK}$ has to be 
$\rho(x^{-1})\Phi_{qK}=\Phi_{qK}$, which means $x\in Sp(n)Sp(1)$. Thus we 
have shown $N(Sp(n)Sp(1))=Sp(n)Sp(1)$.
\end{proof}

\bibliographystyle{plain}

\end{document}